\documentclass[12pt]{amsart}
\usepackage{amssymb, latexsym, amsmath, amscd, array, graphicx, amsfonts}
\usepackage{hyperref}

\swapnumbers \numberwithin{equation}{section}

\usepackage{mathtools}
\usepackage{mathrsfs}
\usepackage{enumerate}
\usepackage{mathabx}
\usepackage{setspace}
\usepackage{color}
\usepackage{comment}
\usepackage{tikz}
\usepackage{tikz-cd}

\theoremstyle{plain}

\newtheorem{thm}{Theorem}[section]

\newtheorem{lemma}[thm]{Lemma}

\newtheorem{prop}[thm]{Proposition}
\newtheorem{cor}[thm]{Corollary}

\theoremstyle{definition}
\newtheorem{defn}[thm]{Definition}

\newtheorem{remark}[thm]{Remark}

\newtheorem{question}[thm]{Question}

 \newcommand{\Wi}{\widetilde}

\def\Int{\protect\operatorname{Int}}

\def\Z{{\mathbb Z}}

\def\1{\hbox{\rm\rlap {1}\hskip.03in{\rom I}}}
\def\Bbbone{{\rm1\mathchoice{\kern-0.25em}{\kern-0.25em}
{\kern-0.2em}{\kern-0.2em}I}}

\long\def\forget#1\forgotten{} %

\newcommand\ver[1]{\marginpar{\tiny Changed in Ver \VER}}

\newcommand{\mc}{ \text {mc}}

\date{\today}

\begin{document}

\title[Positive scalar curvature]{Positive scalar curvature and strongly inessential manifolds}

\author[A.~Dranishnikov]{Alexander  Dranishnikov$^{1}$} 
\thanks{$^{1}$Supported by Simons Foundation}

\address{Alexander N. Dranishnikov, Department of Mathematics, University
of Florida, 358 Little Hall, Gainesville, FL 32611-8105, USA}
\email{dranish@math.ufl.edu}



\begin{abstract}
We prove that a closed $n$-manifold $M$ with positive scalar curvature and abelian fundamental group admits a finite covering $M'$
which is strongly inessential. The latter means that a classifying map $u:M'\to K(\pi_1(M'),1)$ can be deformed to the $(n-2)$-skeleton.
This is proven for all $n$-manifolds with the exception of 4-manifolds with spin universal coverings.
\end{abstract}



\maketitle
\section {Introduction}

The notion of macroscopic dimension was introduced by M.
Gromov~\cite{G2} to study topology of manifolds with a positive scalar curvature  metric.

\begin{defn}
 A metric space $X$ has the macroscopic dimension $\dim_{\mc} X \leq k$ if
there is a uniformly cobounded proper map $f:X\to K$ to a $k$-dimensional simplicial complex.
Then $\dim_{mc}X=m$ where $m$ is minimal among $k$ with $\dim_{\mc} X \leq k$.
\end{defn}
\smallskip
A map of a metric space $f:X\to Y$ is uniformly cobounded if there is a uniform upper bound on the diameter of preimages $f^{-1}(y)$, $y\in Y$.

\

{\bf Gromov's Conjecture.} {\it The macroscopic dimension of the universal covering $\Wi M$ of  a closed  positive scalar curvature  $n$-manifold $M$ satisfies the inequality $\dim_{mc}\Wi M\leq n-2$ for the metric on $\Wi M$ lifted from $M$}.

\

The main examples supporting Gromov's Conjecture are $n$-manifolds of the form $M=N\times S^2$. They admit metrics with PSC in view of
the formula $Sc_{x_1,x_2}=Sc_{x_1}+Sc_{x_2}$ for the Cartesian product $(X_1\times X_2,\mathcal G_1\oplus\mathcal G_2)$
of two Riemannian manifolds $(X_1,\mathcal G_1)$ and $(X_2,\mathcal G_2)$ and the fact that while $Sc_N$ is bounded $Sc_{S^2}$ can be chosen
to be arbitrary large. Clearly, the projection $p:\Wi M=\Wi N\times S^2\to\Wi N$ is a proper uniformly cobounded map to a $(n-2)$-dimensional manifold
which can be triangulated.
Hence, $\dim_{mc}\Wi M\le n-2$.

Since $\dim_{mc}X=0$ for every compact metric space, the
Gromov Conjecture holds trivially for  manifolds with finite fundamental groups. Therefore, Gromov's conjecture  is about manifolds with infinite fundamental groups.
We note that even a weaker version of the Gromov Conjecture that predicts the inequality $\dim_{mc}\Wi M\leq n-1$ for positive scalar curvature manifolds  is out of reach,
since it implies perhaps the famous Gromov-Lawson's Conjecture: {\em A closed aspherical manifold cannot carry a metric of positive scalar curvature}.
The latter is known as a twin sister of the famous Novikov Higher Signature conjecture. Both conjectures are proven only for some tame classes of groups. 
The Gromov Conjecture is proven  in far fewer cases~\cite{B2},\cite{Dr1},\cite{BD2},\cite{DD}.

We note that in the study of topology of positive scalar curvature manifolds it makes sense to consider three different cases: the case of spin manifolds, almost spin manifolds, and totally non-spin
manifolds. The reason for that is the existence of the index theory in the first two cases. Thus, the case of totally non-spin manifolds is the most difficult.
Here we adopt the names {\em almost spin} for manifolds with the spin universal covering and {\em totally non-spin} for manifolds whose universal coverings are non-spin.

Goromov defined {\em inessential manifolds} $M$ as those for which a classifying map $u:M\to B\Gamma$ of the universal covering $\Wi M$ can be deformed to the $(n-1)$-skeleton $B\Gamma^{(n-1)}$ where $n=\dim M$. Clearly, for an inessential $n$-manifold $M$  we have $\dim_{mc}\Wi M\leq n-1$.
In the case of spin manifolds Rosenberg's vanishing index theorem~\cite{R1} implies that  a positive scalar curvature manifold 
with the fundamental group $\Gamma$ satisfying the Analytic Novikov conjecture and 
the Rosenberg-Stolz condition on injectivity of the real K-theory periodization map $per: ko_*(B\Gamma)\to KO_*(B\Gamma)$ is inessential~\cite{BD1}.

Having this in mind we introduce even a stronger version of  Gromov's Conjecture. For that we extend Gromov's definition of inessentiality to the following.
We call an $n$-manifold {\em strongly inessential} if a classifying map of its universal covering $u:M\to B\Gamma$ can be deformed to the $(n-2)$-skeleton.
Answering Gromov's question Bolotov constructed an example of an inessential manifold which is not strongly inessential~\cite{B1}.

\

{\bf Strong Gromov's Conjecture.} {\it  A closed  positive scalar curvature  manifold $M$ admits a strongly inessential finite covering $M'$.}

\

We note that, since the universal cover of $M'$ coincides with the universal cover of $M$, the Strong Gromov Conjecture implies the original one.

In this paper we prove the Strong Gromov Conjecture in the case of abelian fundamental groups. We have one exception here: Our proof does not work 
for almost spin 4-dimensional manifolds. In this case we can show the existence of an inessential finite cover but we cannot prove ts strong inessentiality. 
In the totally non-spin case we were dealing with an opposite problem: We have a technique~\cite{BD2},\cite{DD} to derive the strong inessentiality from the
inessentiality but proving the latter became possible only after recent work of Schoen and Yau~\cite{SY}.

\begin{remark} We note that for spin manifolds the Strong Gromov Conjecture in the case of abelian fundamental group follows from
the main result of~\cite{BD1}. Unfortunately,  Lemma 4.1 in~\cite{BD1}, which is essential for the main result, has a gap in its proof and the attempt 
to fix it  in~\cite{Dr2}, Lemma 6.2 failed as well. The problem there can be reduced to the following question about stable homotopy groups:
\end{remark}
\begin{question}
For which finitely presented  groups $\Gamma$ the natural homomorphism from the coinvariants $\xi:\pi^s_n(E\Gamma^{(n-1)})_\Gamma\to\pi^s_n(B\Gamma^{(n-1)})$ is injective for all $n> 4$ ?
\end{question}

\section {Preliminaries}

We recall that for
every spectrum $E$ there is a connective cover $e \to E$, that is  the spectrum $e$ with the
morphism $ e\to  E$ that induces isomorphisms of homotopy  groups $\pi_i(e) \to \pi_i(E)$ for $i\ge 0$ and
with $\pi_i(e) = 0$ for $i < 0$ . By $KO$ we denote the spectrum for real K-theory, by
$ko$ its connective cover, and by $per : ko\to KO$ the corresponding 
morphism of spectra. We use the standard notation $\pi^s_*$ for the stable homotopy groups.
For a spectrum $E$ we will use  the old-fashioned  notation $ E_*(X)$ for the $E$-homology
of a space $X$.

The following proposition is taken from~\cite{BD1}.

\begin{prop}\label{pi-ko-iso}
	The natural transformation $\pi_*^s(pt)\to ko_*(pt)$ induces an
	isomorphism $\pi_{n}^s(K/K^{(n-2)})\to ko_{n}(K/K^{(n-2)})$ for any CW
	complex $K$.
\end{prop}

We recall that the group of oriented relative bordisms    $\Omega_n(X,Y)$
of  the pair $(X,Y)$ consists of the equivalence classes of pairs $(M,f)$
where $M$ is an oriented $n$-manifold with boundary and $f:(M,\partial M)\to (X,Y)$ is a continuous map.  Two pairs $(M,f)$ and $(N,g)$ are equivalent if there is a pair $(W,F)$, $F:W\to X$  called a {\em bordism} where $W$ is an orientable $(n+1)$-manifold with boundary such that  $\partial W =
M\cup W'\cup N$, $W'\cap M=\partial M$, $W'\cap N=\partial N$, $F|_M=f$,  $F|_N=g$, and $F(W')\subset Y$.

In the special case when $X$ is a point, the manifold $W$ is called a bordism between $M$ and $N$.
The following proposition is proven in~\cite{BD2}.

\begin{prop}\label{bordism} For any CW complex $K$
there is an isomorphism $$\Omega_n(K,K^{(n-2)})\cong H_n(K,K^{(n-2)}).$$
\end{prop}

\

We recall the following classical result~Corollary 6.10.3, \cite{TtD}:
\begin{thm}\label{h-excision}
	Suppose that a CW-complex pair $(X,A)$ satisfies the conditions $\pi_i(X)=0$ for $i\le m$ and $\pi_i(A)=0$ for $i\le m-1$ with $m\ge 2$. Then the quotient map $q:(X,A)\to (X/A,\ast)$
	induces isomorphisms $q_*:\pi_i(X,A)\to\pi_i(X/A)$ for $i\le 2m-1$.
\end{thm}

\section{inessential manifolds}

\begin{defn}
An $n$-manifold $M$  with fundamental group $\Gamma$ is called  {\em inessential} if  a classifying map $u_M:M\to B\Gamma$ of its universal covering 
can be deformed into the $(n-1)$-skeleton  $B\Gamma^{(n-1)}$.
\end{defn}

Note that for an inessential $n$-manifold $M$ we have $\dim_{mc}\Wi M\le n-1$. Indeed, a lift $\Wi{u_M}:\Wi M\to E\Gamma^{(n-1)}$ of a classifying
map is a uniformly cobounded proper map to an $(n-1)$-complex.

Establishing  inessentiality of positive scalar curvature  manifolds  is the first step in a proof of the strong Gromov conjecture. 
We recall that the inessentiality of a manifold can be characterized as follows~\cite{Ba} (see also~\cite{BD1}, Proposition 3.2). 
\begin{thm}\label{ref} Let $M$ be a closed oriented $n$-manifold. 
Then the following are equivalent:

1. $M$ is inessential;

2. $(u_M)_*([M])=0$ in $H_n(B\Gamma)$ where $[M]$ is the fundamental class of $[M]$.  
\end{thm}

In~\cite{BD1} we proved the following addendum to Theorem~\ref{ref}.
\begin{prop}[\cite{BD1}, Lemma 3.5]\label{ref2}
For an inessential manifold $M$  
with a CW complex structure a classifying map $u:M\to B\Gamma$ can be chosen such that 
$$u(M,M^{(n-1)})\subset (B\Gamma^{(n-1)},B\Gamma^{(n-2)}).$$
\end{prop}

We recall that for a manifold to be spin is equivalent to the orientability in any of  the K-theories: complex $KU$, real $KO$, or  their connective covers $ku$ or $ko$ ~\cite{Ru}.

J. Rosenberg connected the realm of positive scalar curvature manifolds to the Novikov Higher Signature conjecture by proving the following~\cite{R1}:
\begin{thm}\label{Ros}
Suppose that the fundamental group $\Gamma$ of a positive scalar curvature spin manifold $M$ satisfies the Strong Novikov conjecture. Then $u_*([M]_{KO})=0$ where $u:M\to B\Gamma$ is a classifying map.
\end{thm}

\

Below we slightly reformulate Theorem 1.3 from~\cite{SY}.
\begin{thm}[Schoen-Yau]\label{SY}
Suppose that a compact oriented $n$-manifold has 1-dimensional integral cohomology classes $\alpha_i$, $i=1,\dots,n-1$ with nontrivial cup product $\alpha_1\smile\dots\smile\alpha_{n-1}\ne 0$.
Then $M$ cannot carry a metric of positive scalar curvature.
\end{thm}

\begin{prop}\label{iness}
A closed orientable $n$-manifold $M$ carrying a metric of positive scalar curvature with  $\pi_1(M)=\mathbb Z^m$  is inessential.
\end{prop}
\begin{proof} Since the case $m<n$ is trivial, we assume that $m\ge n$.
Assume the contrary, $u_*([M])\ne 0$ in $H_*(T^m;\mathbb Z)$ where $u:M\to T^m$ is the classifying map. 
By Proposition 4.6 in~\cite{BD2} there is a projection onto the factor $q:T^m\to T^n$ such that $q_*u_*([M])\ne 0$.
hence $q_*u_*([M])=\ell[T^n]$ with $\ell\ne 0$.
Let $\beta_1,\dots,\beta_n$ be generators of $H^1(T^n;\mathbb Z)$ with $\beta_1\smile\dots\smile\beta_n\ne 0$. Denote by $\alpha_i=(qu)^*(\beta_i)$.
We show that $\alpha_1\smile\dots\smile\alpha_n\ne 0$ to get a contradiction with Schoen-Yau theorem. Note that
$$(qu)_*((\alpha_1\smile\dots\smile\alpha_n)\cap[M]))=(\beta_1\smile\dots\smile\beta_n)\cap q_*u_*([M])\ne 0.$$
Since $(qu)_*$ is an isomorphism of 0-dimensional homology groups, we obtain  $\alpha_1\smile\dots\smile\alpha_n\ne 0$.
\end{proof}

\begin{defn}
An $n$-manifold $M$  with fundamental group $\Gamma$ is called  {\em strongly inessential} if  its classifying map $u_M:M\to B\Gamma$ can be deformed into the $(n-2)$-skeleton  $B\Gamma^{(n-2)}$.
\end{defn}

\section{Spin and almost spin manifolds}

\begin{defn}
We call a discrete group $G$   $p$-{\em tame} if there is a finite covering $\beta:B'\to BG$ that induces  zero homomorphism
$$\beta^*:H^2(BG;\mathbb Z_p)\to H^2(B';\mathbb Z_p).$$
\end{defn}
EXAMPLE. The group $\mathbb Z^n$ is $p$-tame for all $p$. Moreover, any finitely generated abelian group is $p$-tame.
\begin{prop}\label{2-tame}
For any closed almost spin manifold $M$ with 2-tame fundamental group there is a finite cover $p:M'\to M$ with spin $M'$.
\end{prop}
\begin{proof} Let $G=\pi_1(M)$ and let $u_M:M\to BG$ be a classifying map.
Let $p:M\to M$ be the pull-back of $\beta:B'\to BG$ with respect to $u_M$. Then for the Stiefel-Whitney classes we have $w_2(M')=p^*(w_2)$.
It suffices to show that $w_2(M)=u_M^*(\omega)$ for some $\omega\in H^*(BG;\mathbb Z_2)$, then we obtain $w_2(M')=(u_M')^*\beta^*(\omega)=0$.
Since the universal cover $\tilde M$ is spin, it follows that the evaluation of $w_2(M)$ on every spherical cycle is trivial.
In view of the short exact
sequence
$$ \pi_2(M)\to H_2(M)\to H_2(G)\to 0$$ it follows that the
homomorphism $-\cap w_2(M):H_2(M)\to\Z_2$ lies in the image of the
homomorphism $Hom(H_2(G),\Z_2)\to Hom(H_2(M),\Z_2)$. Thus, in the
diagram generated by the universal coefficient theorem exact
sequences
$$
\begin{CD}
0\to Ext(H_1(M),\Z_2) @>i>> H^2(M;\Z_2) @>j>> Hom(H_2(M),\Z_2)\to 0\\
 @Au_1^*=AA @Au_M^*AA @Au^*AA \\
0\to Ext(H_1(G),\Z_2) @>i'>> H^2(G,\Z_2) @>j'>> Hom(H_2(G),\Z_2)\to 0\\
\end{CD}
$$
$j(w_2(M))=u^*(\phi)$ for some $\phi$. Then the diagram chasing implies that $w_2(M)=u_M^*(\omega)$ for $\omega=i'(\alpha)+\bar\phi$ where $\bar\phi$ is arbitrary with $j'(\bar\phi)=\phi$ and $\alpha=(u_1^*)^{-1}(\bar\alpha)$ where $\bar\alpha=i^{-1}(w_2(M)-u_M^*(\bar\phi))$.
\end{proof}

\begin{lemma}\label{2-obstruction}
Suppose that for a closed spin
$n$-manifold $M$, $n>4$, there is a map $u: M\to B\Gamma^{(n-1)}$ that classify its universal cover and has the properties: 
$u(M^{(n-1)})\subset B\Gamma^{(n-2)}$ and $j_*u_*([M]_{ko})=0$ where $j:B\Gamma^{(n-1)}\to B\Gamma^{(n-1)}/B\Gamma^{(n-2)}$ is the qoutient map.
Then $M$ is strongly inessential. 
\end{lemma}
\begin{proof}
We may assume that $M$ has a CW complex structure with one $n$-dimensional cell. Let $\psi:D^n\to M$ be its characteristic map.
By Proposition~\ref{ref2} we may assume that the classifying map $u$ satisfies the condition $u(M^{(n-1)})\subset B\Gamma^{(n-2)}$.
Note that  the homotopy groups of the $(n-1)$-homotopy  fiber $F$ of the inclusion $B\Gamma^{(n-2)}\to B\Gamma^{(n-1)}$ equal the relative $n$-homotopy 
groups, $\pi_{n-1}(F)=\pi_n(B\Gamma^{(n-1)},B\Gamma^{(n-2)})$.
Then the first and the only obstruction to deform $u$ to $B\Gamma^{(n-2)}$ is defined by the cocycle
$c_u:C_n(M)\to\pi_{n-1}(F)$ represented by the composition $$C_n(M)=\pi_n(D^n,\partial D^n)\stackrel{\psi_*}\to\pi_n(M,M^{(n-1)})\stackrel{u_*}\to\pi_n(B\Gamma^{(n-1)},B\Gamma^{(n-2)})$$ with the cohomology class
$o_u=[c_u]\in  H^n(M;\pi_n(B\Gamma^{(n-1)},B\Gamma^{(n-2)}))$. By the Poincare duality with local coefficients, $o_u$ is dual
to the homology class $$PD(o_u)\in H_0(M;\pi_n(B\Gamma^{(n-1)},B\Gamma^{(n-2)}))=\pi_n(B\Gamma^{(n-1)},B\Gamma^{(n-2)})_{\Gamma}$$
represented by $q_*u_*\psi_*(1)$ where $$q_*:\pi_n(B\Gamma^{(n-1)},B\Gamma^{(n-2)})\to\pi_n(B\Gamma^{(n-1)},B\Gamma^{(n-2)})_{\Gamma}$$ 
is the projection onto the group of coinvariants. Note that $\pi_n(B\Gamma^{(n-1)},B\Gamma^{(n-2)})=\pi_n(E\Gamma^{(n-1)},E\Gamma^{(n-2)})$.
Since $n\le 2(n-2)-1$, by 
Theorem~\ref{h-excision}, $$\pi_n(E\Gamma^{(n-1)},E\Gamma^{(n-2)})=\pi_n(E\Gamma^{(n-1)}/E\Gamma^{(n-2)}).$$
It is easy to see that $$\pi_n(E\Gamma^{(n-1)}/E\Gamma^{(n-2)})_\Gamma=\pi_n(B\Gamma^{(n-1)}/B\Gamma^{(n-2)}).$$

Denote by $\bar u:M/M^{(n-1)}=S^n\to B\Gamma/B\Gamma^{(n-2)}$ the induced map. The commutative diagram
$$
\begin{CD}
\pi_n(M,M^{(n-1)}) @>u_*>>\pi_n(B\Gamma^{(n-1)},B\Gamma^{(n-2)}) @>q_*>> \pi_n(B\Gamma^{(n-1)},B\Gamma^{(n-2)})_\Gamma\\
@Ai_*AA @. @V\cong V\bar p_*V\\
\pi_n(D^n/\partial D^n) @>=>>\pi_n(M/M^{(n-1)}) @>\bar u_*>>  \pi_n(B\Gamma^{(n-1)}/B\Gamma^{(n-2)})\\
\end{CD}
$$
implies that
$\bar u_*(1)=\bar p_*q_*u_*\psi_*(1).$  Thus, $\bar u_*(1)=0$ if and only if the obstruction $o_u$ vanishes.

We show that $\bar u_*(1)=0$.
The restriction $n>4$ and Proposition~\ref{pi-ko-iso} imply that $\bar
u_*(1)$ survives to the $ko$-homology group:
$$
\begin{CD}
\pi_n(B\Gamma^{(n-1)}/B\Gamma^{(n-2)}) @>\cong >> \pi_n^s(B\Gamma^{(n-1)}/B\Gamma^{(n-2)}) @>\cong >>ko_n(B\Gamma^{(n-1)}/B\Gamma^{(n-2)}).\\
\end{CD}
$$

Then the  commutative diagram 
$$
\begin{CD}
\pi_n(S^n)  @>\cong>> ko_n(S^n)\\
@V\bar u_*VV @V\bar u_*VV \\
\pi_n(B\Gamma^{(n-1)}/B\Gamma^{(n-2)}) @>\cong>> ko_n(B\Gamma^{(n-1)}/B\Gamma^{(n-2)})
\end{CD}
$$
implies that $\bar u_*(1)= 0$ for $ko_n$ if and only if $\bar u_*(1)=0$ for $\pi_n$. 

From the assumption and the diagram defined by the quotient maps $j':M\to M/M^{(n-1)}=S^n$ and $j:B\Gamma^{(n-1)}\to B\Gamma^{(n-1)}/B\Gamma^{(n-2)}$
$$
\begin{CD}
ko_n(M) @>u_*>>ko_n(B\Gamma^{(n-1)})\\
@V{j'_*}VV @V{j_*}VV\\
ko_n(S^n) @>\bar u_*>>ko_n(B\Gamma^{(n-1)}/B\Gamma^{(n-2)})\\
\end{CD}
$$
it follows that $\bar u_*(1)=\bar u_*j'_*([M]_{ko})=j_*u_*([M]_{ko})=0$.
\end{proof}

\section{K-theory Injectivity Conditions} 

The following is well-known (see 4C, \cite{Ha}).
\begin{prop}\label{per} Let $X$ be an $(n-1)$-connected $(n+1)$-dimensional CW complex.
Then $X$ is homotopy equivalent to the wedge of spheres of dimensions $n$ and $n+1$
together with the Moore spaces $M(\Z_m, n)$.
\end{prop}

We consider the following condition on K-theory of a group $\Gamma$ which appear in our proof of the Strong Gromov Conjecture.

\

(*) There is a classifying CW-complex  $B\Gamma$ such that the inclusion homomorphism
$$
(\phi_n)_*:  KO_*(B\Gamma^{(n)})\to KO_*(B\Gamma)
$$
is injective for all $n>4$.

\

($I$) There is a classifying CW-complex  $B\Gamma$ such that the inclusion homomorphism
$$
(\phi_n)_*:  KO_*(B\Gamma^{(n)})\to KO_*(B\Gamma)
$$
restricted to the image of $KO_*(B\Gamma^{(n-1)})$ is injective for all $n>4$.

\

($\bar I$) There is a classifying CW-complex  $B\Gamma$ such that the inclusion homomorphism
$$
(\bar\phi_n)_*:  KO_*(B\Gamma^{(n)}/B\Gamma^{(n-2)})\to KO_*(B\Gamma/B\Gamma^{(n-2)})
$$
restricted to the image of $KO_*(B\Gamma^{(n-1)})$ is injective for all $n>4$.

\

($\bar I_{(2)}$) There is a classifying CW-complex  $B\Gamma$ such that the inclusion homomorphism
$$
(\bar\phi_n)_*:  KO_*(B\Gamma^{(n)}/B\Gamma^{(n-2)})\otimes\Z_{(2)}\to KO_*(B\Gamma/B\Gamma^{(n-2)})\otimes\Z_{(2)}
$$
restricted to the image of $KO_*(B\Gamma^{(n-1)})\otimes\Z_{(2)}$ is injective for all $n>4$.

\begin{prop}
There are implications
$$
(^*) \Rightarrow I\Rightarrow \bar I\Rightarrow \bar I_{(2)}.
$$
\end{prop}
\begin{proof}
(*) $\Rightarrow I$. Obvious.

\

$I \Rightarrow\bar I$. Consider the commutative diagram defined by exact sequence of pairs
$$
\begin{CD}
KO_*(B\Gamma^{(n-2)}) @>i'_*>> KO_*(B\Gamma^{(n)}) @>j_*'>> KO_*(B\Gamma^{(n)}/B\Gamma^{(n-2)})\\
@V=VV @V(\phi_n)_*VV @V(\bar\phi_n)_*VV\\
KO_*(B\Gamma^{(n-2)}) @>i_*>> KO_*(B\Gamma) @>j_*>> KO_*(B\Gamma/B\Gamma^{(n-2)})\\
\end{CD}
$$
Suppose that $a\in KO_*(B\Gamma^{(n)})$ lies in the image of $KO_*(B\Gamma^{(n-1)})$ and
$(\bar\phi_n)_* (j'_*(a))=0$. We need to show that $j'_*(a)=0$. 
By exactness, there is $b'$ such that $i_*(b')=(\phi_n)_*(a)$. Then $a-i'_*(b')$ lies in the image of $KO_*(B\Gamma^{(n-1)})$ . Since $(\phi_n)_*(a-i'_*((b'))=0$, by the condition $I$, $a=i_*'(b')$. By exactness, $j'_*(a)=0$.

\

$\bar I \Rightarrow\bar I_{(2)}$. Straightforward.
\end{proof}

\

\begin{prop}\label{KO-iness}
Suppose that for a group $\Gamma$ satisfying the condition $\bar I$
a classifying map $u:M\to B\Gamma$  of a closed spin $n$-manifold $M$ takes the KO fundamental class to 0.
Then $M$ is inessential.
\end{prop}
\begin{proof}
We may assume that $u(M)\subset B\Gamma^{(n)}$.
Then by $\bar I$, $u_*([M]_{KO})=0$ in $KO_n(B\Gamma^{(n)})$ and, hence, in
$KO_n(B\Gamma^{(n+1)}/B\Gamma^{(n-1)})$. Let $u'=q\circ u$ where $q:B\Gamma^{(n+1)}\to 
B\Gamma^{(n+1)}/B\Gamma^{(n-1)}$
is the quotient map. In the commutative diagram
$$
\begin{CD}
ko_n(M) @>u'_*>> ko_n(B\Gamma^{(n+1)}/B\Gamma^{(n-1)})\\
@VperVV @V{\cong}VperV\\
KO_n(M) @>u'_*>> KO_n(B\Gamma^{(n+1)}/B\Gamma^{(n-1)})\\
\end{CD}
$$
the homomorphism $per$ is an isomorphism in view of Proposition~\ref{per}. This 
implies that  $u_*'([M]_{ko})=0$ in $ko_n(B\Gamma^{(n)}/B\Gamma^{(n-1)})$ . 
In view of the natural transformation of homology theories $ko_*\to H_*(\ ;\mathbb Z)$  it follows that $u_*'([M])=0$ in $H_n(B\Gamma^{(n+1)}/B\Gamma^{(n-1)})$.
Since the homomorphism $q_*:H_n(B\Gamma^{(n+1)})\to H_n(B\Gamma^{(n+1)}/B\Gamma^{(n-1)})$ is injective, we obtain that $u_*([M])=0$.
Theorem~\ref{ref} completes the proof.
\end{proof}

\begin{prop}\label{strong iness}
Suppose that an inessential manifold $M$ has the fundamental group with property $\bar I_{(2)}$. Then $M$ is strongly inessential.
\end{prop}
\begin{proof}
We may assume that $M$ has a CW complex structure with one $n$-dimensional cell. Let $\psi:D^n\to M$ be its characteristic map.
By Proposition~\ref{ref2} we may assume that the classifying map $u$ satisfies the condition $u(M^{(n-1)})\subset B\Gamma^{(n-2)}$ and $u(M)\subset B\Gamma^{(n-1)}$.
We will show that the lifting problem

\[
  \begin{tikzcd}
  M^{(n-1)} \arrow[r] \arrow[d, "\subset"]& B\Gamma^{(n-2)} \arrow[d, "\subset"] \\
   M \arrow[r,"iu"]\arrow[dotted]{ur} &  B\Gamma^{(n)}.  
  \end{tikzcd}
\]
has a solution. Here $i:B\Gamma^{(n-1)}\to B\Gamma^{(n)}$.  It would  mean that there is a homotopy lift $\hat u: M\to B\Gamma^{(n-2)}$ of $i\circ u$ which agrees with $u$ on $M^{(n-2)}$.
Since $n\ge 4$, the map $\hat u$ induces an isomorphism of the fundamental groups and, hence, is  classifying map. 

We note that a 'simpler' lifting problem
\[
  \begin{tikzcd}
  M^{(n-1)} \arrow[r] \arrow[d, "\subset"]& B\Gamma^{(n-2)} \arrow[d, "\subset"] \\
   M \arrow[r,"u"]\arrow[dotted]{ur} &  B\Gamma^{(n-1)}.  
  \end{tikzcd}
\]
might have no solution.

Note that  the homotopy groups of the $(n-1)$-homotopy  fiber $F$ of the inclusion $B\Gamma^{(n-2)}\to B\Gamma^{(n)}$ equal the relative $n$-homotopy 
groups, $\pi_{n-1}(F)=\pi_n(B\Gamma^{(n)},B\Gamma^{(n-2)})$.
Then the first and the only obstruction to lift $iu$ to $B\Gamma^{(n-2)}$ is defined by the cocycle
$c_u:C_n(M)\to\pi_{n-1}(F)$ represented by the composition
 $$\pi_n(D^n,\partial D^n)\stackrel{\psi_*}\to\pi_n(M,M^{(n-1)})\stackrel{u_*}\to\pi_n(B\Gamma^{(n-1)},B\Gamma^{(n-2)})\stackrel{i_*}\to\pi_n(B\Gamma^{(n)},B\Gamma^{(n-2)})$$ 
with the cohomology class
$o_u=[c_u]\in  H^n(M;\pi_n(B\Gamma^{(n)},B\Gamma^{(n-2)}))$.  By the Poincare Duality with local coefficients, the cohomology class $o_u$ is dual
to the homology class $$PD(o_u)\in H_0(M;\pi_n(B\Gamma^{(n)},B\Gamma^{(n-2)}))=\pi_n(B\Gamma^{(n)},B\Gamma^{(n-2)})_{\Gamma}$$
represented by $q_*i_*u_*\psi_*(1)$ where $$q_*:\pi_n(B\Gamma^{(n)},B\Gamma^{(n-2)})\to\pi_n(B\Gamma^{(n)},B\Gamma^{(n-2)})_{\Gamma}$$ 
is the projection onto the group of coinvariants.  

Note that $\pi_n(B\Gamma^{(n-1)},B\Gamma^{(n-2)})=\pi_n(E\Gamma^{(n-1)},E\Gamma^{(n-2)})$.
Below we will identify these groups. Denote by$$\tilde i:(E\Gamma^{(n-1)},E\Gamma^{(n-2)})  \to (E\Gamma^{(n)},E\Gamma^{(n-2)})$$ is the inclusion induced by $i$.

Since $n\le 2(n-2)-1$, by 
Theorem~\ref{h-excision}, $$\pi_n(E\Gamma^{(n-1)},E\Gamma^{(n-2)})=\pi_n(E\Gamma^{(n-1)}/E\Gamma^{(n-2)}).$$
It is easy to see that $$\pi_n(E\Gamma^{(n-1)}/E\Gamma^{(n-2)})_\Gamma=\pi_n(B\Gamma^{(n-1)}/B\Gamma^{(n-2)}).$$
Similarly, $$\pi_n(E\Gamma^{(n)}/E\Gamma^{(n-1)})_\Gamma=\pi_n(B\Gamma^{(n)}/B\Gamma^{(n-1)}).$$
The homotopy exact sequence of the triple $(E\Gamma^{(n)},E\Gamma^{(n-1)},E\Gamma^{(n-2)})$ brings the following commutative diagram
$$
\begin{CD}
\pi_{n+1}(E\Gamma^{(n)},E\Gamma^{(n-1)})@>>>\pi_{n}(E\Gamma^{(n-1)},E\Gamma^{(n-2)}) @>\tilde i_*>> im(\tilde i_*) @>>> 0\\
@VVV @VVV @Vq_*VV @.\\
\pi_{n+1}(E\Gamma^{(n)},E\Gamma^{(n-1)})_\Gamma@>>>\pi_{n}(E\Gamma^{(n-1)},E\Gamma^{(n-2)})_\Gamma @>\tilde i_*\otimes_\Gamma 1_\Z>> im(\tilde i_*)_\Gamma @>>> 0\\
@V\cong VV @V\cong VV @V\xi VV @.\\
\pi_{n+1}(B\Gamma^{(n)}/B\Gamma^{(n-1)})@>>>\pi_{n}(B\Gamma^{(n-1)}/B\Gamma^{(n-2)}) @> i_*>> im( i_*) @>>> 0\\
\end{CD}
$$
where the row in the middle is exact as obtained by tensor product of the first row  with $\Z$ over $\Z\Gamma$.
By the Five Lemma the homomorphism $\xi$ is an isomorphis.

Denote by $\bar u:M/M^{(n-1)}=S^n\to B\Gamma^{(n-1)}/B\Gamma^{(n-2)}$ the induced map. The commutative diagram
$$
\begin{CD}
\pi_n(M,M^{(n-1)}) @>\tilde i_*u_*>>im(\tilde i_*) @>q_*>> im(\tilde i_*)_\Gamma @.\\
@A\psi_*AA @. @VV\xi V @.\\
\pi_n(D^n,\partial D^n) @>=>>\pi_n(M/M^{(n-1)}) @>i_*\bar u_*>> im(i_*)@>\subset>> \pi_n(B\Gamma^{(n)}/B\Gamma^{(n-2)})\\
\end{CD}
$$
implies that
$i_*\bar u_*(1)=\xi q_*\tilde i_*u_*\psi_*(1).$  Thus, $i_*\bar u_*(1)=0$ if and only if the obstruction $o_u$ vanishes.

We show that $i_*\bar u_*(1)=0$.
The restriction $n>4$ and Proposition~\ref{pi-ko-iso} imply that $\bar
u_*(1)$ survives to the $ko$-homology group:
$$
\begin{CD}
\pi_n(B\Gamma^{(n)}/B\Gamma^{(n-2)}) @>\cong >> \pi_n^s(B\Gamma^{(n)}/B\Gamma^{(n-2)}) @>\cong >>ko_n(B\Gamma^{(n)}/B\Gamma^{(n-2)}).\\
\end{CD}
$$

Then the  commutative diagram 
$$
\begin{CD}
\pi_n(S^n)  @>\cong>> ko_n(S^n)\\
@V\bar u_*VV @V\bar u_*VV \\
\pi_n(B\Gamma^{(n)}/B\Gamma^{(n-2)}) @>\cong>> ko_n(B\Gamma^{(n)}/B\Gamma^{(n-2)})
\end{CD}
$$
implies that $i_*\bar u_*(1)= 0$ for $ko_n$ if and only if $i_*\bar u_*(1)=0$ for $\pi_n$. 

Note that in the diagram
$$
\begin{CD}
ko_n(M/M^{(n-1)}) @>\bar u_*>>ko_n(B\Gamma^{(n-1)}/B\Gamma^{(n-2)}) @>i_*>> ko_n(B\Gamma^{(n)}/B\Gamma^{(n-2)})\\
@V\cong VV  @V\cong VV @V\cong VV\\
KO_n(M/M^{(n-1)}) @>\bar u_*>> KO_n(B\Gamma^{(n-1)}/B\Gamma^{(n-2)})@>i_*>> KO_n(B\Gamma^{(n)}/B\Gamma^{(n-2)}) \\
\end{CD}
$$
the right vertical arrow is an isomorphism by Proposition~\ref{per}. Since the group $KO_n(B\Gamma^{(n-1)}/B\Gamma^{(n-2)})$ is 2-torsion,
from the property $I_{(2)}$ it follows that $i_*\bar u_*(1)=0$ for $KO$. The above diagram implies that $i_*\bar u_*(1)=0$
for $ko$.
\end{proof}

\begin{thm}\label{n}
Suppose that a group $\Gamma$ has the property $ \bar I$ and satisfies the Strong Novikov conjecture.
Then the Strong Gromov conjecture holds for spin $n$-manifolds, $n> 4$, with the fundamental group $\Gamma$.
\end{thm}
\begin{proof}
Let $M$ be a positive scalar curvature spin $n$-manifold.
By Rosenberg's theorem (Theorem~\ref{Ros}) $u_*([M]_{KO})=0$. 
By Proposition~\ref{KO-iness} $ M$ is inessential. By Proposition~\ref{strong iness} $ M$ is strongly inessential.
\end{proof}

\section{Totally non-spin manifolds}

Let  $\nu_M:M\to BSO$ denote
a classifying map for the stable normal bundle of  a  manifold $M$.

The following theorem was proven in~\cite{BD2}.
\begin{thm}\label{main2} Let $M$ be a totally non-spin closed orientable inessential $n$-manifold, $n\geq 5$, whose fundamental group is of the type $FP_3$. 
Then $M$  is strongly inessential. 
\end{thm}

The proof of  Theorem~\ref{main2} uses Wall's theorem~\cite{W} on the cell structure of cobordisms which is known only in dimension $\ge 5$.
In this section we extend this result to $n=4$ using  obstruction theory.

We recall that a map $f:X\to Y$ is called {\em $k$-equivalence} if induces an isomorphism $f_*:\pi_i(X)\to\pi_i(Y)$ for $i<k$ and an epimorphism fro $i=k$.
\begin{prop}\label{local coeff}
Let $(X,Y)$ be a CW pair such that the inclusion $Y\to X$ is 2-equivalence. Then $H_2(X,Y;F)=0$ for any $\pi$-module $F$ where
 $\pi=\pi_1(X)=\pi_1(Y)$ .
\end{prop}
\begin{proof}
We will be using two well-known facts:

(1) A $k$-equivalence, $k>1$, between CW complexes induces an isomorphism of homology groups in dimensions $<k$ for any
local coefficients.

(2) By attaching cells of dimension $k+1,\dots,n$ to $Y$ one can construct a CW complex $A$ and $n$-connected map $f:A\to X$ extending the inclusion $Y\to X$
(Theorem 8.6.1 in~\cite{TtD}).

We consider such $A$ and $f:A\to X$ for $n=3$.
Then the commutative diagram for homology with coefficients in $F$ generated by
$f:(A,Y)\to (X,Y)$,
$$
\begin{CD}
H_2(Y) @>>>H_2(A) @>>> H_2(A,Y) @>>> H_1(Y)@>>> H_1(A) \\
@V=VV @V\cong VV @Vf_*VV @V=VV @V\cong VV\\
H_2(Y) @>>> H_2(X) @>>> H_2(X,Y) @>>> H_1(Y)@>>> H_1(Y)\\
\end{CD}
$$
and the Five Lemma imply that $f_*$ is an isomorphism. Since the inclusion $(A^{(2)},Y^{(2)})\to(A,Y)$ induces an epimorphism of 2-homology with any coefficients,
it follows that $H_2(X,Y;F)=H_2(A,Y;F)=H_2(A^{(2)},Y^{(2)};F)=H_2(Y^{(2)},Y^{(2)};F)=0$.
\end{proof}

We recall that a finitely presented group $\Gamma$ is of type $FP_3$~\cite{Br} if and only if there is a classifying space $B\Gamma $ with finite 3-skeleton $B\Gamma^{(3)}$.

\begin{thm}\label{main1} Let $M$ be a totally non-spin closed orientable inessential 4-manifold,  whose fundamental group is of the type $FP_3$. 
Then $M$  is strongly inessential. 
\end{thm}

\begin{proof}
The  proof  can be broken into four steps: 

(1). Let $\Gamma=\pi_1(M)$. We may assume that $M$ has a CW structure with one $4$-dimensional cell. Since $M$ is inessential, by Proposition~\ref{ref2}, it has a classifying map $u:M\to B\pi^{(3)}$ such that $u( M\backslash D)\subset B\Gamma^{(2)}$, where $D$ is a closed 4-ball $D$ in the 4-dimensional cell of $M$.

(2). Note that the restriction of $u$ to $D$ defines a zero element in $H_4(B\Gamma,B\Gamma^{(2)})$. 
By Proposition~\ref{bordism}, $u|_D$ defines a zero element in $\Omega_4(B\Gamma,B\Gamma^{(2)})$. 
Thus, there is a relative stationary on the boundary bordism
$W'$, $q:W'\to B\Gamma$, between $(D,u|_D)$, $u|_D:(D,\partial D)\to(B\Gamma,B\Gamma^{(2)})$
and  some pair $(N',q')$, $q':N'\to B\Gamma^{(2)}$. We extend $W'$ by the stationaly bordism to a bordism $(W,q)$ between $(M,u)$ and  $(N,q|_N)$.
Then   $q(x,t)=u(x)$ for all $x\in M\setminus D$ and all $t\in[0,1]$. 

(3). First we note that by applying 1-surgery to $int W$ we may assume that  the inclusion $M\to W$ induces an isomorphism of the fundamental groups
$\pi_1(M)\to \pi_1(W)$.

 The induced homomorphism $$(\nu_{W })_*:\pi_2(W)\to\pi_2(BSO)=\mathbb{Z}_2.$$  is surjective
 in view of the total non-spin assumption. Note that every 2-sphere $S$ that generates an element of the kernel of $(\nu_{W_\gamma })_*$
has a trivial stable normal bundle. Since $\pi_1(W)\cong\pi_1(M)$ is a  group of type $FP_3$, $\pi_2(W)$ is a finitely generated $\pi_1(W)$-module (see~\cite{Br}, VIII (4.3)). It follows from Proposition 3.2 and Proposition 3.3 that the kernel of $(\nu_{W})_*$ is finitely generated. Hence we can perform 2-surgery on 
the 5-manifold $\Int W$ to obtain a bordism $\hat W$ between $M$ and $N$ and a map
 $\nu_{\hat{W}}:\hat{W}\longrightarrow BSO$ which induces an isomorphism of 2-dimensional homotopy groups.  Let $i:M\longrightarrow\hat{W}$ denote the inclusion map. Then $(\nu_{\hat{W}})_*\circ i_{*}=(\nu_{M})_*$.
Since  $(\nu_M)_*$ is surjective and $(\nu_{\hat{W}})_*$ is an isomorphism, it follows that $i_{*}:\pi_2(M)\longrightarrow\pi_2(\hat{W})$ is surjective.
Since $B\Gamma$ is aspherical, there is a map $\hat q:\hat W\to B\Gamma$ with $\hat q=q$ on $\partial\hat W$.

(4)  We want to extend the map $\hat q|_N:N\to B\Gamma^{(2)}$ to $\hat W$.

 By the exact sequence of the pairs $(\hat{W},M)$, we get $$\pi_1(\hat{W},M)=\pi_2(\hat{W},M)=0.$$ 
Thus, the inclusion $M\to\hat W$ is a 2-equivalence. We fix a CW complex structure on $\hat W$. By the cellular approximation theorem we may assume that $\hat q(\hat W^{(2)})\subset B\Gamma^{(2)}$.

The first obstruction for this extension lives in $H^3(\hat W,N;\pi_2(B\Gamma^{(2)})).$ By the Poincare-Lefschetz Duality
with twisted coefficients,  $$H^3(\hat W,N;F)=H_2(\hat W, M;F)=0$$ 
in view of Proposition~\ref{local coeff}. Similarly, the second obstruction is trivial,
since $H^4(\hat W,N;F)=H_1(\hat W, M;F)=0$ for any coefficient system.
And the third obstruction lives in $$H^5(\hat W,N;\pi_4(B\Gamma^{(2)}))= H_0(\hat W,M;\pi_4(B\Gamma^{(2)}))=0.$$
Thus, there is an extension of $\hat q|_N$  to a map $g:\hat W\to B\Gamma^{(2)}$. The commutative diagram
$$
\begin{CD}
\pi_1(\hat W) @<\cong<<\pi_1(M)\\
@Vg_*VV @ V\cong VV\\
\pi_1(B\Gamma^{(2)}) @>=>>\pi_1(B\Gamma)\\
\end{CD}
$$
implies that the restriction $g|_M:M\to B\Gamma^{(2)}$ is a classifying map.
 \end{proof}

\

\section{Abelian fundamental group}

For a CW complex $X$ we denote by $cell_k(X)$ the set of $k$-dimensional cells. Let $cell(X)=\coprod_{k\ge 0} cell_k(X)$.
We call a cellular map $f:X\to Y$ between CW-complexes {\em bijective cellular} if there is a preserving dimension bijection $\beta:cell(X)\to cell(Y)$ such that for 
each cell $e\in cell(X)$ there is a closed ball $B\subset e$ such that the restriction $f|_{Int(B)}:Int B\to \beta(e)$ is a homeomorphism.

\begin{prop}
A bijective cellular map $f:X\to Y$ is a homotopy equivalence with a bijective cellular homotopy inverse
$g:Y\to B$. If $\beta$ is the cell bijection for $f$, then $\beta^{-1}$ is the cell bijection for $g$.
\end{prop} 
\begin{proof}
Induction on dimension of $X$.
\end{proof}

Moreover, the following holds true:
\begin{cor}\label{skeleton}
A bijective cellular map $f:X\to Y$ between $n$-dimensional CW complexes is a stratified homotopy equivalence
$$f:(X,X^{(n-1)},\dots, X^{(1)},X^{(0)})\to (Y,Y^{(n-1)},\dots, Y^{(1)},Y^{(0)}).$$

\end{cor}
\begin{prop}\label{step} A bijective cellular map $f_0:X^k\to Y^k$ between $k$-dimensional complexes
	extends to a bijective cellular map
	$f:X\cup_\phi D^{k+1}\to Y\cup_{\phi'}D^{k+1}$ provided $f_0\circ\phi$ is homotopic to $\phi'$.
\end{prop}
\begin{proof}
	We define the map $f$ on $D^{k+1}$ to be a homeomorphism of interior of a ball $D_0^{k+1}\subset D^{k+1}$ to $int D^{k+1}$ extended to $D^{k+1}\setminus int D_0^{k+1}$ by the homotopy between $f_0\circ\phi$ and $\phi'$.
\end{proof}
We consider the minimal CW complex structure on spheres $S^k$.
\begin{prop}
	Suppose that the attaching maps for all cells in a connected CW-complex $X$ are null-homotopic. Then $X$  there is a bijective cellular map $f:X\to\bigvee_{k=1}^{\dim X}\bigvee_{E_k} S^k$. 
\end{prop}
\begin{proof}
	We apply Proposition~\ref{step}	and induction on $\dim X$.
\end{proof}
Let $T^m$ denote the $n$-dimensional torus with the CW-complex structure induced from the CW-complex structure on $S^1=e^0\cup e^1$. By $\Sigma^r$ we denote the $r$-times iterated reduces suspension. Note that $$\Sigma^rX=e^0\cup \bigcup_\alpha e^{r+1}_\alpha\cup\bigcup_\beta e^{r+2}_\beta\cup\dots$$ for the celluar decomposition  of $X$,
$$X=e^0\cup\bigcup_\alpha e^1_\alpha\cup\bigcup_\beta e^2_\beta\dots.$$
\begin{prop}\label{torus}
	The complex $\Sigma^{r-1}T^r$ admits a strict homotopy equivalence $$f:\Sigma^{r-1}T^r\to\bigvee_{k=1}^m\bigvee_{\binom{m}{k}}S^{k+r-1}.$$
\end{prop}
\begin{proof}
	By induction on $r$ we show that all the attaching maps in $\Sigma^\ell T^r$ for $\ell\ge r-1$ are null-homotopic. In view of the fact that $\Sigma(X\times Y)$ is homotopy equivalent to $\Sigma X\vee\Sigma Y\vee\Sigma(X\wedge Y)$~\cite{Ha}, we obtain that $\Sigma^\ell(T^{r-1}\times S^1)$ is homotopy equivalent to
	$\Sigma^\ell T^{r-1}\vee S^{\ell+1}\vee\Sigma^{\ell+1}T^{r-1})$. Then the induction assumption completes the proof.
\end{proof}

Let $T^m_n$ denote the $n$-dimensional skeleton of $T^m$ with respect to the standard CW-complex structure.
\begin{cor}
	The pair $(T^m_n,T^m_\ell)$, $n\ge \ell$, is stably homotopy equivalent to the pair $$(\bigvee_{k=1}^n\bigvee_{\binom{m}{k}}S^k, \bigvee_{k=1}^\ell\bigvee_{\binom{m}{k}}S^k).$$
\end{cor}
\begin{proof}
	Follows from Proposition~\ref{torus} and Corollary~\ref{skeleton}.
\end{proof}

\begin{cor}\label{stablehom}
A finitely generated free abelian group $\mathbb Z^m$ satisfies the K-theory condition (*).
\end{cor}

\

\begin{thm}\label{main}
 The Strong Gromov's Conjecture holds true for closed $n$-manifolds $M$ with  abelian $\pi_1(M)$ for all $n\ne 4$.
For $n=4$ it holds when $M$ is totally no-nspin.
\end{thm}
\begin{proof}
Let $M$ be an almost spin manifold with positive scalar curvature with abelian $\pi_1(M)$. Taking a finite cover of $M$ we may assume that
$\pi_1(M)$ is free abelian. In view of Proposition~\ref{2-tame}, by taking a finite cover, we may assume that $M$ is spin.
By  Corollary~\ref{stablehom} and Theorem~\ref{new}  $M$ is strongly inessential.

Let $M$ is totally non-spin. By Proposition~\ref{iness}, $M$ is inessential. By Theorem~\ref{main1} and Theorem~\ref{main2} $M$ is strongly inessential.

\end{proof}

\begin{question}\label{Q1}
Does the Strong Gromov Conjecture hold for spin 4-manifolds with abelian fundamental group ?
\end{question}

We recall that Bolotov's example $M_b$ of  inessential 4-manifold which is not strongly inessential is spin and has the fundamental group $\mathbb Z\ast\mathbb Z^3$~\cite{B1}.
In view of Question~\ref{Q1} it is  natural to ask   if there is such example with the free abelian fundamental group. If 
the answer is no, then the restriction $n\ne 4$ in Theorem~\ref{main} can be dropped. We note~\cite{B1} that Bolotov's manifold $M_b$ preserves its property after crossing with a circle $S^1$, i.e., the product  $M_b\times S^1$
is inessential but not strongly inessential. In view of the following Proposition and  Theorem~\ref{new}, it 
cannot carry a metric of positive scalar curvature. 

\begin{prop}\label{operations}
(a) Let groups $\Gamma_1$ and $\Gamma_2$ satisfy the property (*), then the free product $\Gamma_1\ast\Gamma_2$ satisfies (*).

(b) Let $\Gamma$ satisfy (*), then $\Gamma\times\mathbb Z$ satisfies (*).
\end{prop}
\begin{proof}
(a) Since $B\Gamma=B\Gamma_1\vee B\Gamma_2$ for $\Gamma=\Gamma_1\ast\Gamma_2$, the fact is obvious.

(b) We consider the product CW structure on $B\Gamma\times S^1=B(\Gamma\times\mathbb Z)$ with the CW complex structure on $B\Gamma$ satisfying (*) and
the standard CW complex structure on $S^1$. In view of the suspension isomorphism, it suffices to check the condition (*) for the reduced suspension $\Sigma(B\Gamma\times S^1)$.
Note that there are homotopy equivalences
$$\Sigma(B\Gamma\times S^1)=
\Sigma^2 B\Gamma\vee \Sigma B\Gamma\vee\Sigma S^1 \ \ \ \text{and} $$
$$\Sigma((B\Gamma\times S^1)^{(k)})=\Sigma( B\Gamma^{(k-1)}\times S^1)\cup\Sigma B\Gamma^{(k)}=
\Sigma^2 B\Gamma^{(k-1)}\vee \Sigma B\Gamma^{(k)}\vee\Sigma S^1.$$
Moreover, the inclusion $\Sigma((B\Gamma\times S^1)^{(k)})\to\Sigma(B\Gamma\times S^1)$ is the wedge of the inclusions
$\Sigma^2 (B\Gamma^{(k-1)})\to \Sigma^2(B\Gamma)$ and $\Sigma (B\Gamma^{(k)})\to\Sigma(\Gamma B)$ plus the identity map on $\Sigma S^1$ where
each of the inclusions is injective in the KO-homology.
\end{proof}

We note that the proof in~\cite{Dr1} of the original Gromov's conjecture for manifolds with the duality fundamental group implicitly  assumes 
that $n> 4$. So the Question~\ref{Q1} is open for the original Gromov's conjecture as well. The latter can be derived from a positive answer to the following

\begin{question}
Does the formula  $$\dim_{mc} (X\times\mathbb R)=\dim_{mc} X+1$$ hold for metric spaces $X$?
\end{question}
In view of a counter-example for the asymptotic dimension~\cite{Dr3}, this formula may not hold for general metric spaces. The spaces of interest 
here are the universal covers of closed manifolds.

\end{document}